\documentclass[11pt,a4paper]{article}
\usepackage[T1]{fontenc}
\usepackage{lmodern}
\usepackage[margin=29mm]{geometry}
\usepackage{amsmath,amssymb,amsthm,mathtools}
\usepackage{microtype}
\usepackage{needspace}
\usepackage{enumitem}
\usepackage{xcolor}
\usepackage[unicode,colorlinks=true,linkcolor=blue!45!black,citecolor=blue!45!black,urlcolor=blue!45!black]{hyperref}
\hypersetup{pdftitle={Local Reflexivity and Exactness for Reduced C*-Algebras of Wreath Products},pdfsubject={Research draft},pdfauthor={Pingchuan Dong and Zhe Dong}}
\numberwithin{equation}{section}
\newtheorem{theorem}{Theorem}[section]
\newtheorem{proposition}[theorem]{Proposition}
\newtheorem{lemma}[theorem]{Lemma}
\newtheorem{corollary}[theorem]{Corollary}
\theoremstyle{definition}
\newtheorem{definition}[theorem]{Definition}
\theoremstyle{remark}
\newtheorem{remark}[theorem]{Remark}
\newcommand{\Cr}{C_r^*}
\newcommand{\Cu}{C_u^*}
\newcommand{\C}{\mathbb C}
\newcommand{\Z}{\mathbb Z}
\newcommand{\N}{\mathbb N}
\newcommand{\F}{\mathbb F}
\newcommand{\T}{\mathbb T}
\newcommand{\wrp}{\wr}
\newcommand{\supp}{\operatorname{supp}}

\newcommand{\id}{\operatorname{id}}

\setlist[enumerate]{label=\textup{(\roman*)},leftmargin=2.4em,itemsep=3pt}
\allowdisplaybreaks[2]

\title{Local Reflexivity and Exactness for\texorpdfstring{\\}{ }Reduced \texorpdfstring{$C^*$}{C*}-Algebras of Wreath Products}
\author{Pingchuan Dong and Zhe Dong}
\date{}

\begin{document}
\maketitle
\begin{abstract}
We prove that, for every countable non-exact discrete group $H$ and every nontrivial countable discrete group $K$, the reduced group $C^*$-algebra $C_r^*(K\wr H)$ is not locally reflexive. These provide the first known examples of non-locally reflexive reduced group $C^*$-algebras, settling a longstanding existence question. Moreover, for every nontrivial countable exact discrete group $K$, we establish that local reflexivity of $C_r^*(K\wr H)$ is equivalent to exactness of $H$, and hence to exactness of $C_r^*(K\wr H)$. This gives a partial affirmative answer to Kirchberg's question of whether local reflexivity implies exactness for $C^*$-algebras. As applications, we construct finitely generated groups with the Haagerup property, as well as finitely generated residually finite groups, whose reduced group $C^*$-algebras are not locally reflexive. In particular, neither the Haagerup property nor coarse embeddability into Hilbert space guarantees local reflexivity of the associated reduced group $C^*$-algebra.
\end{abstract}

\section{Introduction}

Local reflexivity, introduced by Effros and Haagerup \cite{EH}, is a finite-dimensional approximation property of $C^*$-algebras. Exactness, introduced by Kirchberg \cite{K78}, implies local reflexivity \cite{K94}, but whether the converse holds remains open in general \cite{BO,Kirchberg}. Consequently, knowing that a discrete group $G$ is non-exact does not by itself establish that its reduced group $C^*$-algebra $\Cr(G)$ fails local reflexivity. The existence question addressed here is:
\begin{quote}
\emph{Does there exist a countable discrete group $G$ such that $\Cr(G)$ is not locally reflexive?}
\end{quote}
Kirchberg's 2011 list of open problems asks, more specifically, whether $\Cr(G)$ fails local reflexivity when $G$ is a Gromov non-exact group \cite{Kirchberg}.

Failure of local reflexivity was already known for other completions of group algebras. Effros and Haagerup showed that the full group $C^*$-algebra $C^*(\F_2)$ is not locally reflexive \cite{EH}. Ruan and Wiersma subsequently proved failure of local reflexivity for classes of exotic group $C^*$-algebras, including $C_{\ell^p}^*(\F_2)$ for $2<p<\infty$ \cite{RW}. An exotic completion $C_\pi^*(G)$ fits into canonical proper quotient maps
\[
 C^*(G)\twoheadrightarrow C_\pi^*(G)\twoheadrightarrow\Cr(G).
\]
Local reflexivity passes to $C^*$-quotients \cite[Proposition~5.3]{EH}, but its failure need not: $\Cr(\F_2)$ is exact and locally reflexive, although $C^*(\F_2)$ is not \cite{BO,EH}. Thus the full and exotic examples do not answer the existence question for reduced group $C^*$-algebras.

Our main result answers this existence question affirmatively by transferring the obstruction to local reflexivity from uniform Roe algebras to reduced group $C^*$-algebras of wreath products. The construction works for every countable non-exact acting group and every nontrivial countable lamp group, including arbitrary nontrivial finite groups.

For countable discrete groups $K$ and $H$, write $K\wrp H=K^{(H)}\rtimes H$ for the restricted wreath product, where $K^{(H)}$ consists of the finitely supported functions from $H$ to $K$, and $H$ acts by $(h\cdot a)(x)=a(h^{-1}x)$. Our main result is the following.

\medskip
\noindent\textbf{Main Theorem (Theorem~\ref{thm:main}).}
\emph{Let $H$ be a countable non-exact discrete group. For every nontrivial countable discrete group $K$, the reduced group $C^*$-algebra $\Cr(K\wrp H)$ is not locally reflexive.}
\medskip

For each fixed nontrivial countable exact group $K$, this gives the characterization
\[
 \begin{aligned}
 H\text{ is exact}
 &\quad\Longleftrightarrow\quad C_r^*(K\wrp H)\text{ is exact}\\
 &\quad\Longleftrightarrow\quad \Cr(K\wrp H)\text{ is locally reflexive}
 \end{aligned}
\]
for every countable discrete group $H$; see Corollary~\ref{cor:main}. In particular, this gives a partial affirmative answer to Kirchberg's question of whether local reflexivity implies exactness for $C^*$-algebras. 

This characterization concerns the wreath-product family with a fixed exact lamp group. It does not settle whether local reflexivity implies exactness for arbitrary reduced group $C^*$-algebras, nor does it determine local reflexivity of $\Cr(H)$ for general non-exact $H$.

The theorem also separates local reflexivity of reduced group $C^*$-algebras from geometric approximation properties of groups. Combining it with Osajda's non-exact groups having the Haagerup property \cite{OsajdaSC} and the permanence of the Haagerup property under wreath products \cite{CSV}, we obtain finitely generated Haagerup groups $G$ for which $\Cr(G)$ is not locally reflexive (Corollary~\ref{cor:haag}). Thus neither the Haagerup property nor coarse embeddability into Hilbert space implies local reflexivity of the reduced group $C^*$-algebra. Applying the theorem to Osajda's residually finite non-exact groups \cite{OsajdaRF} also gives finitely generated residually finite examples (Corollary~\ref{cor:rf}).

The proof of main theorem (Theorem~\ref{thm:main}) combines operator-algebraic permanence with symbolic coding. Section~\ref{sec:prelim} records the required permanence and crossed-product facts. The starting point of Section~\ref{sec:roe} is Sako's characterization of property A by local reflexivity of the uniform Roe algebra \cite{Sako}. Together with the characterization of exact groups by property A \cite{Ozawa}, it gives, for a proper right-invariant metric on $H$,
\begin{equation}\label{eq:intro-roe}
 \Cu(H)=\ell^\infty(H)\rtimes_r H
 \quad\text{is locally reflexive if and only if $H$ is exact.}
\end{equation}
Although $\Cr(H)$ embeds in $\Cu(H)$, failure of local reflexivity of the larger algebra does not imply failure for the smaller one. We transfer the obstruction by extracting a separable $H$-invariant commutative subalgebra $D\subset\ell^\infty(H)$ such that $D\rtimes_r H$ is not locally reflexive, and realizing this crossed product as a quotient of $\Cr(\Z\wrp H)$. Local reflexivity passes to $C^*$-quotients, so this proves the infinite cyclic lamp case.
The Bernoulli crossed-product description
\begin{equation}\label{eq:intro-bernoulli}
 \Cr(K\wrp H)\cong
 \left(\bigotimes_{h\in H,\min}\Cr(K)\right)\rtimes_r H
\end{equation}
then allows us to pass to general infinite lamps. For infinite $K$, Haar unitaries in $\Cr(K)$, supplied by Thiel's theorem \cite{Thiel}, give an embedding of $\Cr(\Z\wrp H)$ into $\Cr(K\wrp H)$.  

Finite lamps require a different mechanism, since their group $C^*$-algebras contain no copy of $C(\T)$. Adapting Salo's sparse-marker and Hall-matching construction \cite{Salo}, we obtain, for every countable nonamenable $H$, a continuous equivariant surjection from the binary full shift onto the full shift with Cantor alphabet. Pullback along this factor map and passage to reduced crossed products embed an infinite-lamp reduced group $C^*$-algebra into a binary-lamp reduced group $C^*$-algebra. A unital copy of $\C^2$ in every nontrivial finite lamp algebra then yields the finite-lamp case. Section~\ref{sec:coding} constructs the factor map, Section~\ref{sec:finite} completes the proof of the main theorem, and Section~\ref{sec:applications} gives the applications described above.

\vskip 0.3cm
\noindent
{\bf AI Use Statement. } We acknowledge the use of OpenAI’s ChatGPT to assist in exploring mathematical ideas underlying certain results of this paper, particularly in Section ~\ref{sec:coding}. All mathematical arguments have been independently verified and, where necessary, revised by the authors, who take full responsibility for the content and correctness of the final manuscript.

\section{Local reflexivity and crossed products}\label{sec:prelim}

For the reader’s convenience, we recall some of basic terminologies and results about local
reflexivity, crossed products and wreath products in this section. The details can be found in \cite{BO,ER22,Pi03}.

All groups in the paper are countable and discrete, all tensor products of $C^*$-algebras are minimal unless specified otherwise, and all infinite tensor products are taken with respect to the units. We use $\N_0=\{0,1,2,\ldots\}$ for indexing a countable family when zero is convenient.

\begin{definition}
A unital $C^*$-algebra $A$ is \emph{locally reflexive} if for every finite-dimensional operator system $E\subset A^{**}$ the inclusion $E\to A^{**}$ is the point-weak$^*$ limit of unital completely positive maps $E\to A$. The weak$^*$ topology is $\sigma(A^{**},A^*)$. 
\end{definition}

Equivalently, by Effros and Haagerup [4, Theorem 5.1], \(A\) is locally reflexive if and only if, for every finite-dimensional operator space \(F\), every complete contraction \(T:F\to A^{**}\) is the point-weak* limit of a net of complete contractions \(T_\alpha:F\to A\). Thus the operator-system formulation above agrees with the classical notion of local reflexivity for operator spaces.
We recall the following standard facts \cite{BO,EH,EOR,Kirchberg}:
\begin{enumerate}
\item every exact $C^*$-algebra is locally reflexive;
\item local reflexivity passes to $C^*$-subalgebras and $C^*$-quotients;
\item a $C^*$-algebra is locally reflexive if and only if all its separable $C^*$-subalgebras are locally reflexive.
\end{enumerate}
We spell out quotient permanence because the direction of this implication is important here.

\begin{lemma}\label{lem:quotient-lr}
Let $q:A\twoheadrightarrow B$ be a unital surjective $*$-homomorphism. If $A$ is locally reflexive, then $B$ is locally reflexive.
\end{lemma}
\begin{proof}
The normal map $q^{**}:A^{**}\to B^{**}$ identifies $B^{**}$ with a central summand $zA^{**}$. Let $\theta:B^{**}\to zA^{**}$ be the inverse identification and choose a state $\omega$ on $B^{**}$. Then
\[
 j(b)=\theta(b)+\omega(b)(1-z)
\]
is a unital complete order embedding with $q^{**}j=\id_{B^{**}}$. If $E\subset B^{**}$ is a finite-dimensional operator system, local reflexivity of $A$ supplies unital completely positive maps $\phi_i:j(E)\to A$ converging point-weak$^*$ to the inclusion. The maps $q\phi_i j|_E:E\to B$ converge to the inclusion of $E$ in $B^{**}$ under the weak$^*$ topology $\sigma(B^{**},B^*)$, as seen by testing against $B^*$ and its pullback under $q$.
\end{proof}

\begin{lemma}\label{lem:crossed}
Suppose that $H$ acts on unital $C^*$-algebras $A$ and $D$ and that $\varphi:A\to D$ is an equivariant unital $*$-homomorphism. There is an induced $*$-homomorphism
\[
 \varphi\rtimes_r\mathrm{id}_H:A\rtimes_r H\longrightarrow D\rtimes_r H.
\]
It is injective when $\varphi$ is injective and surjective when $\varphi$ is surjective. 
\end{lemma}
\begin{proof}
The coefficientwise map on finite crossed-product sums extends to the reduced completions by the regular-representation construction; see \cite{BO}. If $\varphi$ is surjective, its induced map has dense range. The range of a $*$-homomorphism is closed, so the induced map is surjective. If $\varphi$ is injective, use the faithful canonical conditional expectations $E_A$ and $E_D$. They satisfy
\[
 E_D\bigl((\varphi\rtimes_r\mathrm{id}_H)(x)\bigr)=\varphi(E_A(x)).
\]
Applying this identity to $x^*x$ proves injectivity.
\end{proof}

For a group $N$ with an action of $H$, the canonical identifications are
\begin{equation}\label{eq:semidirect}
 \Cr(N\rtimes H)\cong\Cr(N)\rtimes_r H,
 \qquad
 \Cr(K\wrp H)\cong
 \left(\bigotimes_{h\in H}\Cr(K)\right)\rtimes_r H.
\end{equation}
These formulas do not require amenability of $N$ or $K$. We will also use that exactness of discrete groups is preserved under subgroups, finite direct products, directed unions, and extensions \cite{KW}. A group is exact precisely when its reduced group algebra is exact.

\section{From the uniform Roe algebra to infinite lamps}\label{sec:roe}

Equip $H$ with a proper right-invariant metric arising from a proper length function. With the left translation action $\alpha_h(f)(x)=f(h^{-1}x)$, the standard realization is
\[
 \Cu(H)=\ell^\infty(H)\rtimes_r H.
\]
The right-invariant metric makes the left translation operators finite-propagation operators. Property A and exactness of $H$ are equivalent \cite{Ozawa}; Sako's theorem \cite[Theorem~1.1]{Sako} gives
\begin{equation}\label{eq:sako}
 H\text{ non-exact}\quad\Longrightarrow\quad
 \Cu(H)\text{ not locally reflexive}.
\end{equation}

\begin{lemma}\label{lem:separable-roe}
If $H$ is non-exact, there is a separable unital $H$-invariant commutative $C^*$-subalgebra $D\subset\ell^\infty(H)$, generated by countably many projections, such that $D\rtimes_r H$ is not locally reflexive.
\end{lemma}
\begin{proof}
By \eqref{eq:sako} and separable determination, choose a separable non-locally reflexive $C^*$-subalgebra $B\subset\Cu(H)$. Adjoining the ambient unit preserves this failure, so assume that $B$ is unital. Choose a countable dense set $\{b_j:j\ge1\}$ in $B$. For every $j,m\ge1$, choose a finite sum
\[
 x_{j,m}=\sum_{h\in F_{j,m}} f_{j,m,h}\lambda_h,
 \qquad \|b_j-x_{j,m}\|<1/m.
\]
Only countably many coefficient functions occur. Each bounded function on $H$ is uniformly approximable by finite-valued functions, and hence belongs to the unital algebra generated by a countable collection of characteristic projections. Collect these projections for all coefficients, together with all their $H$-translates, in a countable set $\mathcal P$. Put $D=C^*(1,\mathcal P)$. Then $D$ is separable and $H$-invariant, and every $x_{j,m}$ belongs to $D\rtimes_r H$. Lemma~\ref{lem:crossed} gives isometric inclusions
\[
 B\subset D\rtimes_r H\subset\Cu(H).
\]
Local reflexivity of the middle algebra would imply that of $B$, a contradiction.
\end{proof}

\begin{lemma}\label{lem:one-unitary}
A unital commutative $C^*$-algebra generated by countably many projections is generated by one unitary.
\end{lemma}
\begin{proof}
List the projections as $p_1,p_2,\ldots$, with repetitions if necessary, and put
\[
 f=\sum_{n\ge1}3^{-n}p_n.
\]
The series converges in norm and $0\le f\le1/2$. The map
\[
 \{0,1\}^{\N}\longrightarrow[0,1/2],
 \qquad (\varepsilon_n)\longmapsto\sum_{n\ge1}3^{-n}\varepsilon_n
\]
is continuous and injective: at the first differing coordinate, its contribution exceeds the sum of all later contributions. Its inverse on its compact image is continuous. Each $p_n$ is consequently a continuous function of $f$, so $C^*(1,f)$ is the whole algebra. Since the exponential is injective on $[0,1/2]$, the unitary $u=e^{if}$ generates the same unital algebra.
\end{proof}

\begin{proposition}\label{prop:cyclic-quotient}
If $D$ is a unital commutative $C^*$-algebra generated by one unitary and $H$ acts on $D$, there is a surjective $*$-homomorphism
\[
 \Cr(\Z\wrp H)\twoheadrightarrow D\rtimes_r H.
\]
\end{proposition}
\begin{proof}
Let $u$ generate $D$, and let $v_h$ denote the standard unitary at the $h$-coordinate of $\Z^{(H)}$. The commuting family $\{\alpha_h(u):h\in H\}$ defines a unital $*$-homomorphism
\[
 q:C^*(\Z^{(H)})\longrightarrow D,
 \qquad q(v_h)=\alpha_h(u).
\]
It is equivariant and surjective because $q(v_e)=u$. Since $\Z^{(H)}$ is abelian, its full and reduced group $C^*$-algebras agree. Apply Lemma~\ref{lem:crossed} and \eqref{eq:semidirect}.
\end{proof}

\begin{theorem}\label{thm:cyclic}
If $H$ is non-exact, then $\Cr(\Z\wrp H)$ is not locally reflexive.
\end{theorem}
\begin{proof}
Choose $D$ as in Lemma~\ref{lem:separable-roe}. Lemma~\ref{lem:one-unitary} and Proposition~\ref{prop:cyclic-quotient} realize $D\rtimes_r H$ as a quotient of $\Cr(\Z\wrp H)$. Thus Lemma~\ref{lem:separable-roe} and Lemma~\ref{lem:quotient-lr} imply that  $\Cr(\Z\wrp H)$ is not locally reflexive.
\end{proof}

\begin{proposition}\label{prop:infinite-embedding}
For every infinite group $K$ and every group $H$, there is a unital embedding
\[
 \Cr(\Z\wrp H)\hookrightarrow\Cr(K\wrp H).
\]
\end{proposition}
\begin{proof}
Since $K$ is infinite, it follows from \cite[Proposition~8.1]{Thiel} that the canonical trace on $\Cr(K)$ admits a Haar unitary $w$, meaning that $\tau_K(w^n)=0$ for every nonzero integer $n$. Since the spectral distribution of \(w\) with respect to \(\tau_K\) is the normalized Haar measure on \(\mathbb T\), we have \(\sigma(w)=\mathbb T\). Hence continuous functional calculus yields a unital embedding \(C(\mathbb T)\cong C_r^*(\mathbb Z)\hookrightarrow C_r^*(K)\).
Tensor this same embedding over all coordinates of $H$. The resulting embedding is equivariant for the Bernoulli actions. Apply Lemma~\ref{lem:crossed} and \eqref{eq:semidirect}.
\end{proof}

\begin{corollary}\label{cor:infinite-lamps}
If $H$ is non-exact and $K$ is infinite, then $\Cr(K\wrp H)$ is not locally reflexive.
\end{corollary}
\begin{proof}
Use Proposition~\ref{prop:infinite-embedding}, Theorem~\ref{thm:cyclic}, and permanence under $C^*$-subalgebras.
\end{proof}
To show that the above result is also true for every nontrivial finite group $K$,  we need the following section to establish the symbolic coding theorem.

\section{Nonlinear compression of the binary full shift}\label{sec:coding}

For a compact space $B$, the full shift $B^H$ carries the action
\[
 (g\cdot x)(h)=x(g^{-1}h).
\]
A \emph{factor map} means a continuous equivariant surjection. In this section the word metric is left-invariant, $d(g,h)=|g^{-1}h|$.  It is independent of the right-invariant metric used for the Roe algebra realization in Section~\ref{sec:roe}.

The purpose of this section is to overcome the finite-dimensional nature of a
finite lamp algebra. The argument for infinite lamps used a copy of $C(\T)$
inside a single lamp algebra; this is unavailable for a finite lamp.

\subsection*{Background for the coding argument}

The proof of Theorem~\ref{thm:finite-factor} mainly uses graph-theoretic combinatorics, especially bipartite matching via Hall’s theorem. These tools are assembled into a local encoding–decoding construction in symbolic dynamics, closely related to cellular automata on groups.
We recall the concepts and elementary results used in
Theorem~\ref{thm:finite-factor}, together with their roles in its proof.
The construction stores each required output letter in a recognizable binary
block at a nearby site. Separation keeps the blocks from interfering;
nonamenable expansion supplies enough possible storage sites; and Hall's
theorem assigns a different site to each required letter.
\vskip 0.2cm
\noindent
{\bf Finite patterns and local rules.}
For a finite alphabet $A$, give $A$ the discrete topology and $A^H$ the
product topology. A \emph{pattern} on a finite set $W\subset H$ is an element
of $A^W$. A basic open set in $A^H$, called a \emph{cylinder set}, prescribes
a pattern on finitely many coordinates. If $\kappa:A^W\to B$ is a fixed
function, then
\[
 \Phi(x)(g)=\kappa\bigl((x(gw))_{w\in W}\bigr)
\]
defines a continuous equivariant map $A^H\to B^H$: each output coordinate
uses finitely many input coordinates, and translating $x$ simply translates
the window $gW$. Indeed, for $k\in H$ the rule gives
$\Phi(k\cdot x)(g)=\Phi(x)(k^{-1}g)$.
Thus, once a finite decoding rule is defined on \emph{every} input pattern,
the remaining requirement for a factor map is surjectivity.

\vskip 0.2cm
\noindent
{\bf Nonamenable expansion.}
The right multiplication form of F\o lner's criterion \cite{Folner} says that
$H$ is amenable if and only if, for every finite $P\subset H$ and every
$\varepsilon>0$, some nonempty finite $F\subset H$ satisfies
$|FP|<(1+\varepsilon)|F|$. Hence a nonamenable group has a finite set $P$
and $\delta>0$ such that $|FP|\ge(1+\delta)|F|$ for all nonempty finite $F$.

\vskip 0.2cm
\noindent
{\bf The finite Hall marriage theorem.}
Let $X$ and $Y$ be the two vertex classes of a finite bipartite graph, and
write $N(E)\subset Y$ for the set of all neighbors of $E\subset X$.
Hall's theorem \cite{Hall} (see \cite[Section~1.4]{LM24}) states that an injective assignment
$m:X\to Y$ with $m(x)$ adjacent to $x$ for every $x\in X$ exists if and
only if
\[
 |N(E)|\ge |E|\qquad\text{for every }E\subset X.
\]

\begin{theorem}\label{thm:finite-factor}
If $H$ is nonamenable and $B$ is a nonempty finite alphabet, there is a factor map
\[
 \{0,1\}^H\twoheadrightarrow B^H.
\]
\end{theorem}

\begin{proof}
First assume that $H$ is finitely generated. Choose $a\in H$ with $a^2\ne e$; such an element exists, since a group in which every element has square $e$ is abelian and hence amenable. By the failure of the F\o lner condition, choose a finite symmetric generating set $S$ containing $e$ and $a$ such that
\begin{equation}\label{eq:expansion}
 |FS|\ge2|F|\qquad\text{for every finite $F\subset H$}.
\end{equation}
Write $U_j=S^j$ and $s=|S|$. Then
\begin{equation}\label{eq:ball-growth}
 |FU_j|\ge2^j|F|,\qquad 2^j\le|U_j|\le s^j.
\end{equation}
The role of nonamenability is precisely this uniform expansion estimate.

{\bf\smallskip\noindent\emph{Binary blocks.}}
Each block has two parts: a fixed marker, which identifies its center; and
freely chosen data bits, which carry information. 
For an integer $r\ge2$, put $K=\{a^{-1},e,a,a^2\}$. On $U_r\setminus K$, join distinct vertices $q,q'$ when $q'=qa$ or $q=q'a$. This finite graph has maximum degree at most two, so a greedy choice gives an independent set $Q$ satisfying
\begin{equation}\label{eq:payload} 
Q\subseteq U_r\setminus K,\qquad Q\cap Qa=\varnothing,\qquad |Q|\ge\frac{|U_r|-4}{3}.
\end{equation}

To construct $Q$, repeatedly choose any remaining vertex $q$ in
$U_r\setminus K$, add it to $Q$, and remove $q$ together with its remaining
neighbors $qa$ and $qa^{-1}$. Stop when no vertices remain. No two chosen
vertices are adjacent, so $Q\cap Qa=\varnothing$. Each step removes at most
three vertices, giving $|U_r\setminus K|\le 3|Q|$. Since $|K|\le4$, this yields
the bound in \eqref{eq:payload}.

Consider the binary patterns on $U_r$ that equal $1$ at $e,a$, are arbitrary on $Q$, and vanish elsewhere. There are $2^{|Q|}$ such patterns. After extension by zero outside $U_r$, each satisfies
\begin{equation}\label{eq:center-test}
 p(t)=p(ta)=1\quad\Longleftrightarrow\quad t=e.
\end{equation}
Indeed, independence excludes pairs within $Q$, the exclusion of $K$ prevents pairs between $Q$ and $\{e,a\}$, and $a^2\ne e$ makes $e$ the unique center of the pair $\{e,a\}$.

Set
\begin{equation}\label{eq:radii}
 D=2r+2,\qquad R=D+\lceil\log_2|U_D|\rceil.
\end{equation}

Here $D$ is the required separation between block centers: radius-$r$ blocks
with this separation leave a buffer that prevents a marker from straddling
two blocks. The larger radius $R$ bounds how far a letter may be moved from
its output site to its storage center. Its definition ensures
$2^{R-D}\ge |U_D|$, which is exactly the balance needed in the Hall estimate
below.
Since $R\le(2r+2)(1+\log_2s)+1$, the lower bound $|Q|\ge(2^r-4)/3$ grows exponentially in $r$, whereas $\log_2(|B|\,|U_R|)\le\log_2|B|+R\log_2s$ grows at most linearly. Fix $r$ large enough that
\begin{equation}\label{eq:capacity}
 2^{|Q|}\ge |B|\,|U_R|.
\end{equation}
Assign distinct such patterns $p_{b,o}$ to the pairs $(b,o)\in B\times U_R$, and fix a default letter $b_0\in B$.

{\bf\smallskip\noindent\emph{Local decoding.}}
We now fix the decoder before choosing any particular preimage. At an output position $g$, a site $t\in gU_R$ votes for $b$ if, for some $(b,o)\in B\times U_R$,
\begin{equation}\label{eq:vote}
 x(th)=p_{b,o}(h)\quad(h\in U_r),\qquad to=g.
\end{equation}
Define $\Phi(x)(g)$ to be the voted letter when there is exactly one vote, and $b_0$ otherwise. Since \(U_RU_r=U_{R+r}\), the value \(\Phi(x)(g)\) depends only on the coordinates of \(x\) in the finite set \(gU_{R+r}\). Thus there is a fixed local rule \(\kappa:\{0,1\}^{U_{R+r}}\to B\) such that \(\Phi(x)(g)=\kappa((x(gu))_{u\in U_{R+r}})\). Consequently, \(\Phi\) is continuous in the product topology. Moreover, for all \(k,g\in H\), we have \(\Phi(k\cdot x)(g)=\Phi(x)(k^{-1}g)=(k\cdot\Phi(x))(g)\), proving \(H\)-equivariance.

{\bf\smallskip\noindent\emph{Surjectivity.}}
It is enough to realize configurations that differ from $b_0$ at finitely
many sites, since these configurations are dense in the full shift. For such
a finite set $F$, we seek one distinct storage center within distance $R$ of
each $g\in F$. A maximal separated set supplies well-spaced candidates.
Hall's condition checks that every subset of output sites has enough
candidate centers, so that the assignments can be made simultaneously.
Let $F\subset H$ be finite and let $y\in B^H$ equal $b_0$ outside $F$. Choose a maximal $D$-separated subset $T$ of the finite set $FU_{R-D}$. For every $E\subset F$, maximality implies
\[
 EU_{R-D}\subset (T\cap EU_R)U_D,
\]
and hence
\begin{equation}\label{eq:hall-estimate}
 |T\cap EU_R|\ge\frac{|EU_{R-D}|}{|U_D|}
 \ge\frac{2^{R-D}|E|}{|U_D|}\ge|E|.
\end{equation}
The finite Hall marriage theorem therefore gives an injection $m:F\hookrightarrow T$ with $d(g,m(g))\le R$.

Indeed, maximality makes the $D$-neighborhoods of the candidates cover
$EU_{R-D}$, and each candidate covers at most $|U_D|$ sites. Expansion gives
at least $2^{R-D}|E|$ sites to cover. Hence at least $|E|$ candidates lie in
$EU_R$, as required by Hall's theorem for the relation
$g\sim t$ if $t\in T\cap gU_R$. The matching $m$ is used only to build a
preimage of the chosen target configuration.

Place $p_{y(g),\,m(g)^{-1}g}$ on each $m(g)U_r$ and fill all remaining coordinates with $0$. The offsets belong to $U_R$. Distinct blocks are at distance at least $D-2r=2$, so no pair $t,ta$ of $1$'s can cross between blocks. By \eqref{eq:center-test}, the sites satisfying $x(t)=x(ta)=1$ are exactly $m(F)$. Thus every codebook match occurs at a prescribed site, where it recovers the assigned letter and offset. Each $m(g)$ votes precisely at $g$, giving $\Phi(x)=y$.
The image of $\Phi$ therefore contains all configurations equal to $b_0$ outside a finite set. These are dense in $B^H$, and the image is compact and hence closed, so $\Phi$ is surjective.

Finally let $H$ be countable and nonamenable, without finite generation. It contains a finitely generated nonamenable subgroup $L$: otherwise all its finitely generated subgroups would be amenable, and so would their directed union. The construction for $L$ has a finite local rule
\[
 \phi(z)(\ell)=\kappa\bigl((z(\ell f))_{f\in E}\bigr),
 \qquad E\subset L\text{ finite}.
\]
Define on $H$ the same rule $\Phi(x)(h)=\kappa((x(hf))_{f\in E})$. It is $H$-equivariant. For each left coset $cL$, choose an $L$-preimage of the target configuration $\ell\mapsto y(c\ell)$, independently of the other cosets. These choices combine into an $H$-preimage. Thus $\Phi$ is surjective.
\end{proof}

\begin{remark}\label{rem:salo}
The proof adapts the sparse-marker and Hall-matching construction of Salo \cite{Salo}. We give a self-contained proof for binary input and arbitrary finite output alphabets; no theorem from \cite{Salo} is used as a black box. The proof above explicitly permits binary input and any prescribed finite output alphabet: the extra information costs only the factor $|B|$ in \eqref{eq:capacity}. We use fixed two-point markers, a separate data region, and a buffer between blocks to make recognition precise. No invariant-measure lifting or preservation of Bernoulli measure is asserted. \end{remark}

\begin{lemma}\label{lem:iterate}
Let $X$ be a compact $H$-space. If there is a factor map $(\alpha,\beta):X\to X\times X$, then
\[
 \Psi:X\longrightarrow X^{\N_0},
 \qquad \Psi(x)_n=\alpha(\beta^n(x)),
\]
is a factor map for the diagonal action.
\end{lemma}
\begin{proof}
Continuity and equivariance hold coordinatewise. Given $y_0,\ldots,y_n\in X$, choose $x_n$ with $\alpha(x_n)=y_n$, and recursively choose $x_j$ with $(\alpha(x_j),\beta(x_j))=(y_j,x_{j+1})$. Then $x_0$ realizes the desired first $n+1$ coordinates. For an infinite target sequence, the closed fibers prescribing the first $n+1$ coordinates are nonempty and nested. Compactness gives a point in their intersection.
\end{proof}

\begin{corollary}\label{cor:cantor-factor}
For every nonamenable $H$, there is a factor map
\[
 \{0,1\}^H\twoheadrightarrow(\{0,1\}^{\N_0})^H.
\]
\end{corollary}
\begin{proof}
Take $B=\{0,1\}^2$ in Theorem~\ref{thm:finite-factor}, identify $B^H$ with $X\times X$ for $X=\{0,1\}^H$, and apply Lemma~\ref{lem:iterate}. The coordinate-rearrangement homeomorphism $X^{\N_0}\cong(\{0,1\}^{\N_0})^H$ is equivariant.
\end{proof}

\section{Finite lamps and the main theorem}\label{sec:finite}

In this section, the symbol $C_p$ denotes the cyclic group of order $p$, where $p$ is a
prime number:
\[
  C_p=\Z/p\Z=\{\overline0,\overline1,\ldots,\overline{p-1}\},
  \qquad
  \overline r+\overline s=\overline{r+s}.
\]
Thus the operation is addition modulo $p$.

\begin{proposition}\label{prop:binary-embedding}
Let $A=\bigoplus_{n\ge0}C_2$. For every nonamenable $H$, there is a unital embedding
\[
 \Cr(A\wrp H)\hookrightarrow\Cr(C_2\wrp H).
\]
\end{proposition}
\begin{proof}
Pontryagin duality gives
\[
 \widehat A=\{0,1\}^{\N_0},\qquad
 \Cr(A^{(H)})=C((\widehat A)^H).
\]
Corollary~\ref{cor:cantor-factor} induces, by pullback, an injective equivariant $*$-homomorphism
\[
 C((\widehat A)^H)\hookrightarrow C(\{0,1\}^H)
 =\Cr(C_2^{(H)}).
\]
Apply Lemma~\ref{lem:crossed} and \eqref{eq:semidirect}.
\end{proof}

\begin{lemma}\label{lem:finite-lamp-embedding}
For every nontrivial finite group $K$ and every group $H$, there is a unital embedding
\[
 \Cr(C_2\wrp H)\hookrightarrow\Cr(K\wrp H).
\]
\end{lemma}
\begin{proof}
The averaging projection
\[
 e_K=\frac1{|K|}\sum_{k\in K}\lambda_k\in\Cr(K)
\]
is neither $0$ nor $1$. Thus $(a,b)\mapsto ae_K+b(1-e_K)$ embeds $\C^2\cong\Cr(C_2)$ unitally into $\Cr(K)$. Tensor the same map over $H$ and use Lemma~\ref{lem:crossed}.
\end{proof}

\begin{theorem}\label{thm:main}
Let $H$ be a countable non-exact discrete group. For every nontrivial countable discrete group $K$, the reduced group $C^*$-algebra $\Cr(K\wrp H)$ is not locally reflexive.
\end{theorem}

\begin{proof}
The case of infinite $K$ is Corollary~\ref{cor:infinite-lamps}. Suppose that $K$ is finite and nontrivial. Since amenable groups are exact, the non-exact group $H$ is nonamenable. Proposition~\ref{prop:binary-embedding} and Lemma~\ref{lem:finite-lamp-embedding} give
\[
 \Cr(A\wrp H)\hookrightarrow
 \Cr(C_2\wrp H)\hookrightarrow\Cr(K\wrp H),
 \qquad A=\bigoplus_{n\ge0}C_2.
\]
The first algebra is not locally reflexive by Corollary~\ref{cor:infinite-lamps}, because $A$ is infinite. Permanence under subalgebras proves the claim.
\end{proof}

\begin{corollary}\label{cor:main}
Fix a nontrivial countable exact group $K$. For every countable discrete group $H$, the following conditions are equivalent:
\begin{enumerate}
\item $H$ is exact;
\item $C_r^*(K\wrp H)$ is exact;
\item $\Cr(K\wrp H)$ is locally reflexive.
\end{enumerate}
In particular, the equivalence holds for every nontrivial finite lamp group.
\end{corollary}

\begin{proof}
If $K$ and $H$ are exact, then $K^{(H)}$ is exact as a directed union of finite direct products of $K$. Exactness is preserved under extensions, so $K\wrp H$ is exact. Exactness of its reduced group $C*$-algebra implies local reflexivity. Conversely, local reflexivity forces $H$ to be exact by Theorem~\ref{thm:main}. 
\end{proof}

\section{Examples and permanence}\label{sec:applications}

\subsection{Finitely generated residually finite examples}

If $K$ and $H$ are finitely generated, so is $K\wrp H$: generators of $H$, together with generators of the lamp at the identity coordinate, generate the wreath product. Osajda constructs finitely generated residually finite non-exact groups \cite{OsajdaRF}. Fix one such group $H_{\mathrm{RF}}$.

\begin{corollary}\label{cor:rf}
For every $d\ge1$ and every prime $p$, the groups
\[
 \Z^d\wrp H_{\mathrm{RF}},\qquad C_p\wrp H_{\mathrm{RF}}
\]
are finitely generated and residually finite, and their reduced group $C^*$-algebras are not locally reflexive.
\end{corollary}
\begin{proof}
Theorem~\ref{thm:main} proves failure of local reflexivity. To check residual finiteness, let $A$ be abelian and residually finite and let $H$ be residually finite. For a nonidentity element $(a,h)\in A\wrp H$, a finite quotient of $H$ separates it if $h\ne e$. If $h=e$, choose a finite quotient $H\to Q$ that is injective on the finite set $\supp(a)$, and a finite quotient $A\to A_0$ that preserves a nonidentity lamp value. Folding lamp values over fibers gives a homomorphism
\[
 A\wrp H\longrightarrow A_0\wrp Q.
\]
Abelianness makes this folding multiplicative, and injectivity on the support prevents cancellation of the selected value. The target is finite. Apply this to $A=\Z^d$ or $C_p$.
\end{proof}

\subsection{The Haagerup property and coarse embeddability}

Osajda constructs finitely generated groups acting properly on CAT(0) cubical complexes and lacking property A \cite[Theorem~2]{OsajdaSC}. Fix such a group $H_{\mathrm{Haag}}$. It has the Haagerup property and is non-exact. Cornulier, Stalder, and Valette prove that standard wreath products of countable Haagerup groups again have the Haagerup property \cite[Theorem~1.1]{CSV}.

\begin{corollary}\label{cor:haag}
For every prime $p$, the finitely generated groups
\[
 \Z\wrp H_{\mathrm{Haag}},\qquad C_p\wrp H_{\mathrm{Haag}}
\]
have the Haagerup property and coarsely embed into Hilbert space, while their reduced group $C^*$-algebras are not locally reflexive.
\end{corollary}
\begin{proof}
Both lamp groups have the Haagerup property, so \cite[Theorem~1.1]{CSV} applies. A proper affine isometric action on a Hilbert space gives a coarse embedding for a finitely generated group. Theorem~\ref{thm:main} proves the remaining assertion.
\end{proof}

These examples show that coarse embeddability, or even the Haagerup property, does not imply local reflexivity of the reduced group algebra.

\subsection{Passing to larger groups and amenable-kernel extensions}

\begin{proposition}\label{prop:permanence}
Suppose that $\Cr(W)$ is not locally reflexive.
\begin{enumerate}
\item If $W\le G$, then $\Cr(G)$ is not locally reflexive.
\item If $N\trianglelefteq G$ is amenable and $G/N\cong W$, then $\Cr(G)$ is not locally reflexive.
\end{enumerate}
\end{proposition}
\begin{proof}
For (i), use the canonical inclusion $\Cr(W)\hookrightarrow\Cr(G)$. For (ii), amenability gives $1_N\prec\lambda_N$. Inducing to $G$ shows that the regular representation of $G/N$, inflated to $G$, is weakly contained in $\lambda_G$. It therefore induces a surjection $\Cr(G)\twoheadrightarrow\Cr(G/N)$. Now use quotient permanence.
\end{proof}

Finally, the main theorem, Theorem~\ref{thm:main} gives no conclusion about local reflexivity of $\Cr(H)$ for a general non-exact group $H$. In fact, for nontrivial countable exact lamps $K$, Corollary~\ref{cor:main} establishes the equivalence between the non-local reflexivity of $\Cr(K\wrp H)$ and the non-exactness of the countable discrete group $H$.

\vskip 0.5cm
\noindent
Pingchuan Dong\\
Department of Electrical Engineering\\
Columbia University\\
NewYork, NY 10027, USA\\
E-mail address: pd2827@columbia.edu

\vskip 0.5cm

\noindent
Zhe Dong\\
School of Mathematical Sciences\\
Zhejiang University\\
Hangzhou 310027, China\\
E-mail address: dongzhe@zju.edu.cn

\end{document}